\documentclass{amsart}
\usepackage{amsmath}
\usepackage{amssymb}
\usepackage{amsthm}
\usepackage{amsaddr}
\usepackage{mathrsfs}
\usepackage[all,cmtip]{xy}
\usepackage{times}
\usepackage{graphicx}
\usepackage{multicol}
\makeatletter
\newcommand{\eqnum}{\refstepcounter{equation}\textup{\tagform@{\theequation}}}
\makeatother

\def\dom{{\mathbf d}}

\def\J{{\mathcal J}}

\def\P{\vec{{\mathscr P}}}

\def\ran{{\mathbf r}}

\def\S{\mathcal S}

\def\<{\langle}
\def\>{\rangle}
\def\leq{\leqslant}
\def\geq{\geqslant}

\newcommand{\twobar}{/\kern-0.3em/}

\newcommand{\upset}[1]{(#1)^{\uparrow}}

\numberwithin{equation}{section}

\newtheorem{theorem}{Theorem}[section]
\newtheorem{prop}[theorem]{Proposition}
\newtheorem{lemma}[theorem]{Lemma}
\newtheorem{cor}[theorem]{Corollary}

\theoremstyle{definition}
\newtheorem{remark}[theorem]{Remark}
\newtheorem{example}[theorem]{Example}

\newtheorem{definition}[theorem]{Definition}

\newcommand{\ams}[2]{{\footnotesize\noindent AMS 2000 \textit{Mathematics subject
classification:} Primary #1\\[-2pt]\phantom{\footnotesize\noindent AMS 2000
\textit{Mathematics subject classification:}} Secondary #2\vspace{1pc}}}

\newcommand{\nobraupset}[1]{#1^{\uparrow}}

\numberwithin{equation}{section}

\title{Closed inverse subsemigroups of graph inverse semigroups}

\author{Amal AlAli \& N.D. Gilbert}

\address{
School of Mathematical and Computer Sciences\\
and the Maxwell Institute for the Mathematical Sciences,\\
Heriot-Watt University, Edinburgh EH14 4AS, U.K.\\
asa80@hw.ac.uk, N.D.Gilbert@hw.ac.uk}
\date{}
\begin{document}

\thispagestyle{empty}
\maketitle
\begin{abstract}
As part of his study of representations of the polycylic monoids, M.V. Lawson described all the closed inverse submonoids
of a polycyclic monoid $P_n$ and classified them up to conjugacy.  We show  that Lawson's description can be extended to closed inverse subsemigroups of graph inverse semigroups. We then apply B. Schein's theory of cosets in inverse semigroups to the closed inverse subsemigroups of graph inverse semigroups: we give 
necessary and sufficient conditions for a closed inverse subsemigroup of a graph inverse semigroup to have finite index, and determine 
the value of the index when it is finite.
\end{abstract}

\ams{20M18}{20M30, 05C25}

\section{Introduction}
Graph inverse semigroups were introduced by Ash and Hall \cite{AshHall}: the construction, recalled in detail in section 
\ref{grinvsgp}, associates to any directed graph $\Gamma$ an inverse semigroup $\S(\Gamma)$ whose elements are
pairs of directed paths in $\Gamma$ with the same initial vertex.  If $\Gamma$ has a single vertex and $n$ edges with $n>1$,
then $\S(\Gamma)$ is the polycyclic monoid $P_n$ as defined by Nivat and Perrot \cite{NivPer}: if $n=1$ then $\S(\Gamma)$ is the
bicyclic monoid $B$ with an adjoined zero.  Ash and Hall give necessary and sufficient conditions on the structure of $\Gamma$ for $\S(\Gamma)$ to be congruence-free, and they use graph inverse semigroups to study the realisation of finite posets as the posets of $\J$--classes in finite semigroups.  The structure of  graph inverse semigroups as HNN extensions of inverse semigroups with zero was presented in \cite[section 5]{DomGil}. For more recent work on the structure of graph inverse semigroups, we refer to
\cite{JonLaw, MesMitch, MMMP}.  Connections between graph inverse semigroups and graph $C^*$--algebras have been fruitfully 
studied in \cite{Pat}.

As part of his study of representations of the polycylic monoids, Lawson \cite{LwPoly} described all the closed inverse submonoids
of a polycyclic monoid $P_n$ and classified them up to conjugacy.  We show in section \ref{cliss} that Lawson's description can be extended to closed inverse subsemigroups of graph inverse semigroups.  As in Lawson's study, there are three types: finite chains of idempotents, infinite chains of idempotents, and closed inverse subsemigroups of {\em cycle type} that are generated (as closed
inverse subsemigroups) by a single non-idempotent element.  In section \ref{clissindex} we apply Schein's theory of cosets in inverse subgroups \cite{Sch} to the closed inverse subsemigroups of graph inverse semigroups as classified in section \ref{cliss}: we give 
necessary and sufficient conditions for a closed inverse subsemigroup $L$ of $\S(\Gamma)$ to have finite index, and determine 
the value of the index when it is finite.

\section{Preliminaries}

\subsection{Cosets}
\label{cosets}
Let $S$ be an inverse semigroup with semilattice of idempotents $E(S)$.  We recall that the \emph{natural partial order} on $S$ is defined by
\[ s \leq t \Longleftrightarrow \text{there exists} \; e \in E(S) \; \text{such that} \; s=et \,.\]
A subset $A \subseteq S$ is \emph{closed} if, whenever $a \in A$ and $a \leq s$, then $s \in A$.  The closure $\nobraupset{B}$ of a subset
$B \subseteq S$ is defined as
\[ \nobraupset{B} = \{ s \in S : s \geq b \; \text{for some} \; b \in B \}\,.\]
A subset $L$ of $S$ is \emph{full} if $E(S) \subseteq L$.

Let $L$ be a closed inverse subsemigroup of $S$, and let $t \in S$ with $tt^{-1} \in L$.  Then the subset
\[ \upset{Lt} = \{ s \in S : \text{there exists $x \in L$ with} \; s \geq xt \} \]
is a (right) coset of $L$ in $S$.  For the basic theory of such cosets we refer to \cite{Sch}: the essential facts that we require are
contained in the following result.

\begin{prop}{\cite[Proposition 6.]{Sch}}
\label{cosetsofL}
Let $L$ be a closed inverse subsemigroup of $S$.
\begin{enumerate}
\item Suppose that $C$ is a coset of $L$.  Then $\upset{CC^{-1}}=L$.
\item If $t \in C$ then $tt^{-1} \in L$ and $C = \upset{Lt}$.  Hence two cosets of $L$ are either disjoint or they coincide.
\item Two elements $a,b \in S$ belong to the same coset $C$ of $L$ if and only if $ab^{-1} \in L$.
\end{enumerate}
\end{prop}

We note that the cosets of $L$ partition $S$ if and only if $L$ is full in $S$.
The cardinality of the set of cosets of $L$ in $S$ is the {\em index} of $L$ in $S$, denoted by $[S:L]$.  

The closed inverse submonoids of free inverse monoids were completely described by Margolis and Meakin in \cite{MarMea}.
For other related work on inverse subsemigroups of finite index, see \cite{AlAliGil} and the first author's PhD thesis \cite{amal_thesis}.

\subsection{Graph inverse semigroups}
\label{grinvsgp}
Let $\Gamma$ be a finite directed graph with vertex set $V(\Gamma)$ and edge set $E(\Gamma)$. Let $\P(\Gamma)$ be the
path category of $\Gamma$, with source and target maps $\dom$ and $\ran$.  We note that $\P(\Gamma)$ admits \emph{empty} (or \emph{length zero}) paths that consists of a single vertex.
The {\em graph inverse semigroup} $\S(\Gamma)$ of $\Gamma$ has underlying  set
$$\{ (v,w) : v,w \in \P(\Gamma) \,, \dom(v)=\dom(w) \} \cup \{ 0 \}$$
equipped with the binary operation
\[
(t,u)(v,w) = 
\begin{cases}
(t,pw) & \text{if $u=pv$ in $\P(\Gamma)$}, \\
(pt,w) & \text{if $v=pu$ in $\P(\Gamma)$}, \\
0 & \text{otherwise.}
\end{cases} \]
This composition is illustrated in the following diagrams:

\begin{center}
\includegraphics[height=5cm]{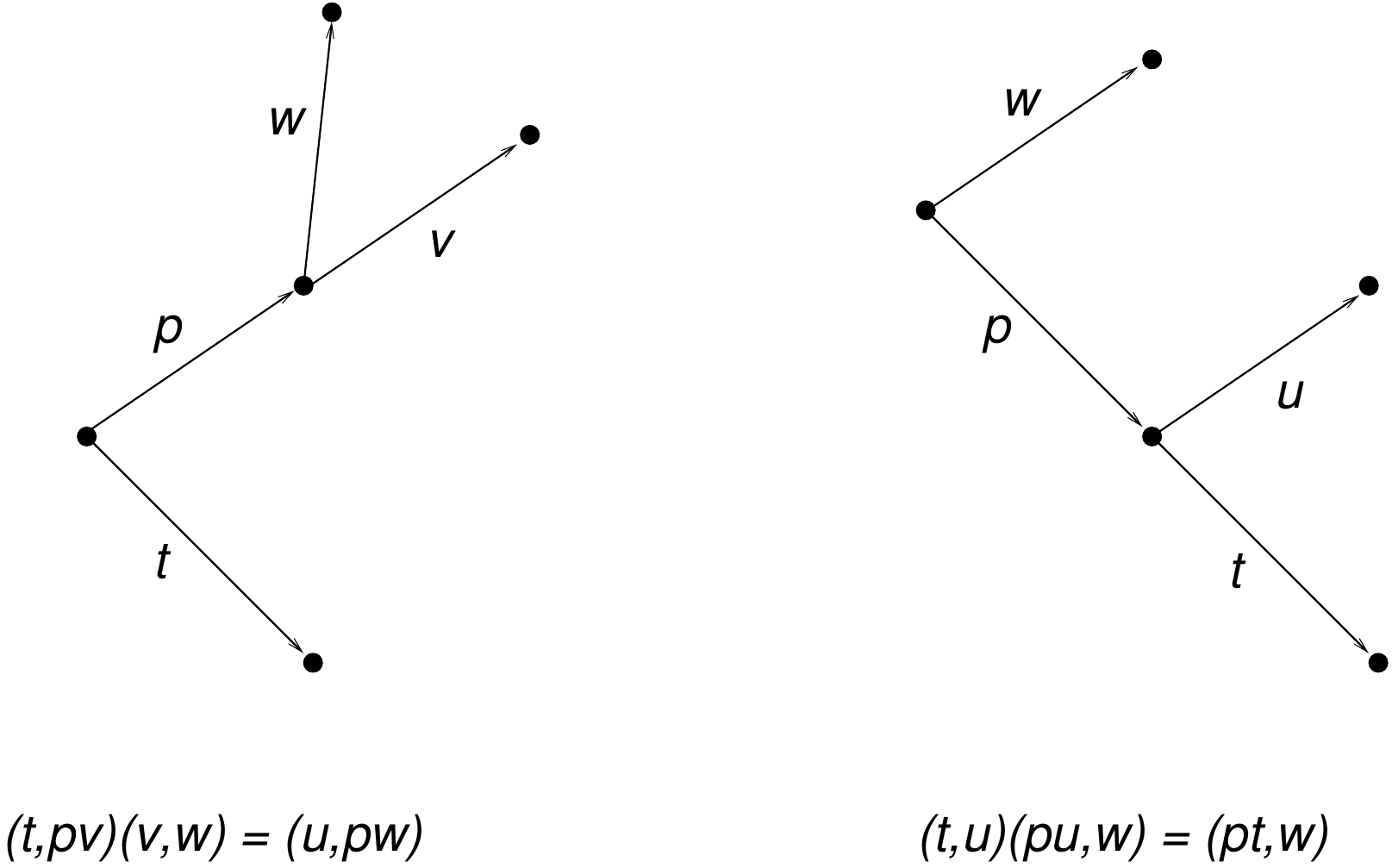}
\end{center}

The inverse of $(v,w)$ is given by $(v,w)^{-1} = (w,v)$.
The idempotents of $\S(G)$ are the pairs $(u,u)$ and $0$: if we
identify $E(\S(G))$ with $\P(G) \cup \{ 0 \}$, then
$\P(G) \cup \{ 0 \}$ becomes a semilattice with ordering given by
\begin{equation}
\label{grinvord}
u \leq v \quad \text{if and only if} \; v \; \text{is a suffix of} \;  u
\end{equation}
and composition (meet)
\begin{equation}
\label{grinvmeet}
u \wedge v = \begin{cases}
u & \text{if $v$ is a suffix of $u$},\\
v & \text{if $u$ is a suffix of $v$},\\
0 & \text{otherwise} \end{cases} \end{equation}
Hence $u \wedge b$ is non-zero if and only if one of $u,v$ is a suffix of the other: in this case we say that $u,v$ are {\em suffix comparable}.

The natural partial order on non-zero elements of $\S(G)$ is then given by $(t,u) \leq (v,w)$ if and only if there exists a path $p \in \Gamma$ such
that $t=pv$ and $u=pw$: that is, we descend in the natural partial order from $(v,w)$ by prepending the same prefix to each of $v$ and $w$., and ascend from $(v,w)$ by deleting an identical prefix from each of $v$ and $w$.  Recall that an inverse semigroup $S$ with zero $0 \in S$ is
said to be $E^*$--{\em unitary}, if whenever $e \in E(S), e \ne 0$ and $s \in S$ with $s \geq e$ then $s \in E(S)$.  It is clear that
graph inverse semigroups are $E^*$--unitary.  For further structural results about graph inverse semigroups, we refer to \cite{JonLaw,MesMitch}.
Graph inverse semigroups as topological inverse semigroups have been recently studied in \cite{MMMP}.

\section{Closed inverse subsemigroups of graph inverse semigroups}
\label{cliss}
Our first result generalizes -- and closely follows --  Lawson's classification \cite[Theorem 4.3]{LwPoly} of closed inverse submonoids of the polycyclic monoids $P_n$ to the closed inverse subsemigroups of graph inverse semigroups $\S(\Gamma)$.   Given Lawson's
insights, the generalization is largely routine, but it is perhaps slightly surprising that the classification extends from bouquets
of circles (giving the polycyclic monoids as graph inverse semigroups) to arbitrary finite directed graphs, and so we have presented it
in detail.  Our notational conventions also differ slightly from those in \cite{LwPoly}.

\begin{theorem}
\label{clisms_of_grinv}
In a graph inverse semigroup $\S(\Gamma)$ there are three types of proper closed inverse subsemigroups $L$:
\begin{enumerate}
\item Finite chain type: $L$ consists of a finite chain of idempotents.
\item Infinite chain type: $L$ consists of an infinite chain of idempotents.
\item Cycle type: $L$  has the form
\[
L = L_{p,d} = \{ (vp^{r}d,vp^{s}d) : r,s \geq 0 \; \text{with} \; v \; \text{a suffix of}\; p \}
\cup \{(q,q) : q \; \text{a suffix of }d \}, \]
where $p$ is a directed circuit in $\Gamma$, $d$ is a directed path in $\Gamma$ starting at the initial point of $p$, and where $p,d$ do
not share a non-trivial prefix.  In this case, $L$ is the smallest closed inverse subsemigroup of $\S(\Gamma)$ containing $(d,pd)$.
\end{enumerate} 
\end{theorem}

\begin{proof}
It is easy to see that closed chains of idempotents are indeed closed inverse subsemigroups.  For the cycle type, if
$L = L_{p,d}$ then any two elements are suffix comparable, and we have
\begin{align*}
(q,q)(q',q') &= (q \wedge q', q \wedge q') \; \text{for suffixes $q,q'$ of $d$}, \\
(q,q) (vp^{r}d,vp^{s}d) &=  (vp^{r}d,vp^{s}d) =  (vp^{r}d,vp^{s}d)(q,q).
\end{align*}
Now consider $ (vp^{r}d,vp^{s}d) (wp^{j}d,wp^{k}d)$: write $p=v_0v$ and suppose that $s<j$.  Then
$wp^jd = wp^{j-s-1}v_0vp^sd$ and so
\[ (vp^{r}d,vp^{s}d) (wp^{j}d,wp^{k}d) = (wp^{j-s-1}v_0vp^{r}d,wp^{k}d) = (wp^{j-s+r}d,wp^kd) \in L \,, \]
and a similar calculation applies if $s>j$.  If $s=j$ and $v$ is a suffix of $w$, say $w=v_1v$, then
\begin{align*}
 (vp^{r}d,vp^{s}d) (wp^{s}d,wp^{k}d) &= (vp^{r}d,vp^{s}d) (v_1vp^{s}d,wp^{k}d) \\
&= (v_1vp^{r}d,wp^{k}d) \\ &= (wp^{r}d,wp^{k}d) \in L \end{align*}
and a similar calculation applies if $s=j$ and $w$ is a suffix of $v$.  Hence $L$ is a subsemigroup of $\S(\Gamma)$,
and since the inverse of an element of $L$ is clearly also in $L$, we deduce that $L$ is an inverse subsemigroup of
$\S(\Gamma)$.  Since we ascend in the natural partial order in $L$ by deleting identical prefixes from the paths $vp^{r}d$ and $vp^{s}d$,  or from a given suffix of $d$, it is also clear that $L$ is closed.

If $F$ is a closed inverse subsemigroup of $\S(\Gamma)$ and contains $(d,pd)$, then for any $m,n \geq 0$ we have
$(pd,d)^m(d,pd)^n = (p^md,d)(d,p^nd) = (p^md,p^nd) \in F$.  Ascending in the natural partial order, we may obtain any 
element of $L_{p,d}$, and so $L_{p,d} \subseteq F$.

Let $L$ be a closed inverse subsemigroup of $\S(\Gamma)$.  If $w$ and $w'$ are paths occurring in elements of $L$ and are not
suffix comparable, then the product of the idempotents $(w,w)$ and $(w',w')$ in $L$ is equal to $0$, and so $0\in L$ and by closure $L=\S(\Gamma)$.  Hence if $L$ is proper, any two paths occurring in elements of $L$ are suffix comparable and hence have the same terminal vertex.  By definition, if $(u,v) \in \S(\Gamma)$ then $u,v$ have the same initial vertex:  hence if $(u,v) \in L$ then $u,v$ have the same initial and the same terminal vertex in $\Gamma$.  Suffix comparability then ensures that any proper closed inverse subsemigroup of $\S(\Gamma)$ consisting entirely of idempotents is either a finite or an infinite chain.  We note that in the second case, in order to obtain directed paths of arbitrary length, $\Gamma$ must contain a directed circuit.

We shall now describe those closed inverse subsemigroups of $\S(\Gamma)$ which contain non-idempotent elements. Suppose that 
$L \neq E(L)$ is a closed inverse subsemigroup of $\S(\Gamma)$.
Then there exists $(u,v) \in L,$ with $u \neq v$, and we may assume that the path  $u$ is shorter than the path $v$.  Hence $u$ is a suffix of $v$ and so $v=pu$ for some path $p$.  Since $u$ and $v$ have the same initial and terminal vertices, $p$ must be a directed
circuit in $\Gamma$.  If $p$ and $u$ share a common prefix, with $p=ap_1$ and $u=au_1$ then 
\[ (u_1,p_1u) \geq (au_1,ap_1u) = (u,pu) \]
and so by closure, $(u_1,p_1u) \in L$.

Amongst the non-idempotent elements $(u,pu) \in L$, choose $u=d$ to have smallest possible length, and then having chosen $d$,
choose $p$ to be a non-empty directed circuit of smallest possible length.  Then $d,p$ do not share a non-trivial prefix.   Now for 
any $m \geq 0$ we have $(p^md,p^md) \in E(L)$ and so, if $(w_1,w_2) \in L$, each $w_i$ is a suffix of some directed path
$p^{m_i}d$.  Since $L$ is closed, every suffix of $d$ is in $L$ and by minimality of $|d|$, every element of $L$ that contains a suffix of $d$ is an idempotent $(q,q)$.  Hence if $|w_i| < |d|$ we have $w_1=w_2$.  So we may now assume that for $i=1,2$ we have
$|w_i| \geq |d|$, and so $w_1 = up^rd$, $w_2=vp^sd$ for some $r,s \geq 0$ and suffixes $u,v$ of $p$.
\end{proof}

From the result of the previous Theorem, we may immediately conclude the following: 

\begin{cor}
If the graph $\Gamma$ contains no directed circuit, then every proper closed inverse subsemigroup of $\S(\Gamma)$ is a chain
of idempotents.
\end{cor}

Our next result, based on \cite[Theorem 4.4]{LwPoly} which treats the polycyclic monoids, classifies the closed inverse subsemigroups of a graph inverse semigroup up to conjugacy.  We begin with the following definitions

\begin{definition}
Let $L =\upset{u,u}$ be a closed inverse subsemigroup of finite chain type in a graph inverse semigroup $\S(\Gamma)$.
We call the initial vertex of the directed path $u$ the \textit{root} of $L$.  
\end{definition}

Adapting ideas of \cite[Section 1.3]{Loth} from words to paths in $\Gamma$:

\begin{definition}
Two paths $p,q$ in $\Gamma$ are \textit{conjugate} if there are paths $u,v$ in $\Gamma$ such that $p=uv$ and $q=vu$. 
Equivalently, (see \cite[Proposition 1.3.4]{Loth}) there exists a path $w$ in $\Gamma$ such that $wp=qw$.  Conjugate paths 
must be directed circuits in $\Gamma$, and conjugacy amounts to the selection of an alternative initial edge.
\end{definition}

The following Lemma is due to Lawson and is extracted from the proof of \cite[Theorem 4.4]{LwPoly}.

\begin{lemma}
\label{cong_in_E}
Let $S$ be an $E^*$--unitary inverse semigroup.
If $H$ and $K$ are conjugate closed inverse subsemigroups of $S$ with $H \ne S \ne K$ and $H \subseteq E(S)$ then $K \subseteq E(S)$.
Moreover, if $H$ has a minimum idempotent, then so does $K$.
\end{lemma}

\begin{proof}
There exists $s \in S$ with $s^{-1}Hs \subseteq K \; \text{and} \; sKs^{-1} \subseteq H$.  Let $k \in K$: then $k \ne 0$ and
$sks^{-1} \in H$ and so $sks^{-1} \in E(S)$.  It follows that $s^{-1}(sks^{-1})s=(s^{-1}s)k(s^{-1}s) \in E(S)$
and $(s^{-1}s)k(s^{-1}s) \leq k$.  Since $S$ is $E^*$--unitary, we deduce that $k \in E(S)$.

Now suppose that $m \in H \subseteq E(S)$ is the minimum idempotent and that $e \in K$.  Then $m \leq ses^{-1}$ and so
\[ s^{-1}ms \leq s^{-1}ses^{-1}s = es^{-1}s \leq e \]
and so $s^{-1}ms$ is a minimum idempotent in $K$.
\end{proof}

\begin{theorem} \label{cong_in_Gamma}
\leavevmode
\begin{enumerate}
\item Let $L$ be a closed inverse subsemigroup of  $\S(\Gamma)$ of finite chain type. Then all closed inverse subsemigroups conjugate to $L$ are of finite chain type. 
Two closed inverse subsemigroups $L =\upset{u,u}$ and $K =\upset{v,v}$ are conjugate in $\S(\Gamma)$ if and only if they
 have the same root.
\item Let $L$ be a closed inverse subsemigroup of  $\S(\Gamma)$ of infinite chain type. Then all closed inverse subsemigroups conjugate to $L$ are also of infinite chain type. Two closed inverse subsemigroups of infinite chain type are conjugate if and only if there are idempotents $(s,s)\in L$ and $(t,t)\in K$ such that for all paths $p$ in $\Gamma,$ we have that $(ps,ps)\in L$ if and only if 
$(pt,pt)\in K$.  
\item Let $L$ be a closed inverse subsemigroup of  $\S(\Gamma)$ of cycle type. The only closed inverse subsemigroups conjugate to $L$ are of cycle type. Moreover, $L_{p,d}$ is conjugate to $L_{q,k}$ if and only if $p$ and $q$ are conjugate directed circuits in $\Gamma$.
\end{enumerate}
\end{theorem}

\begin{proof}
(a)  It follows from Lemma \ref{cong_in_E} that if $L$ has finite chain type then so does every closed inverse subsemigroup conjugate to 
$L$ .

Suppose that $L$ and $K$ have the same root $x_0 \in V(\Gamma)$.  Then $(u,v) \in \S(\Gamma)$, and for any suffix $w$ of
$u$ we have
\[ (v,u)(w,w)(u,v) = (v,v) \in K \,.\]
Similarly, for any suffix $t$ of $v$, $(u,v)(t,t)(v,u) = (u,u) \in L$.  Hence $L$ and $K$ are conjugate.

Conversely, suppose that $L$ and $K$ are conjugate, with conjugating element $(p,q) \in \S(\Gamma)$, so that for any suffixes
$w$ of $u$ and $t$ of $v$ we have
\[ (q,p)(w,w)(p,q) \in K \; \text{and} \; (p,q)(t,t)(q,p) \in L \,.\]
Then $(q,p)(\ran(u),\ran(u))(p,q) \in K$, so that $p$ and $\ran(u)$ are suffix-comparable:
hence $p$ also ends at $\ran(u)$, and $(q,p)(\ran(u),\ran(u))(p,q)=(q,q) \in K$.  Therefore $q$ is a suffix of $v$.  Similarly,
$p$ is a suffix of $u$.

Let $v=v_1q$: then
\[ (p,q)(v,v)(q,p) = (p,q)(v_1q,v_1q)(q,p) = (v_1p,v_1p) \in L \]
and so $v_1p$ is a suffix of $u$.  Let $u=u_0v_1p$: then
\[ (q,p)(u,u)(p,q) = (q,p)(u_0v_1p,u_0v_1p)(p,q) = (u_0v_1q,u_0v_1q) \in K\]
and so $u_0v_1q$ is a suffix of $v$.  But $v=v_1q$ and so $u_0$ is a vertex (namely the root of $L$), and $u=v_1p$.
Hence $u$ and $v$ have the same initial vertex, and so $L$ and $K$ have the same root.

(b)  By Lemma \ref{cong_in_E} any closed inverse subsemigroup\ $K$ conjugate to $L$ must be of chain type, and by part (a) $K$ must be infinite.
Suppose that $(t,s)L(s,t) \subseteq K \; \text{and} \; (s,t)K(t,s) \subseteq L$.  Since $0 \not\in K$ we have, for all $(u,u) \in L$,
that $s$ is suffix comparable with $u$ and similarly for all $(v,v) \in K$, that $t$ is suffix comparable with $v$.  If we consider 
$u$ with $|u| \geq |s|$ then $s$ must be a suffix of $u$ and by closure of $L$ we have $(s,s) \in L$.  Similarly $(t,t) \in K$.
Suppose that $(ps,ps) \in L$.  Then
$(t,s)(ps,ps)(s,t) = (pt,pt) \in K$ and similarly if $(pt,pt) \in K$ then $(ps,ps) \in L$.

Conversely, if $s$ and $t$ exist as in the Theorem and $(w,w) \in L$ then $s$ is suffix comparable with $w$. \\ If
$w$ is a suffix of $s$, with $s=hw$, then
\[ (t,s)(w,w)(s,t) = (t,hw)(w,w)(hw,t) = (t,t) \in K \]
and if $s$ is a suffix of $w$ with $w=ps$ then $(ps,ps) \in L$ and so
\[ (t,s)(w,w)(s,t) = (t,s)(ps,ps)(s,t) = (pt,pt) \in K \,.\]
Similarly $(s,t)K(t,s) \subseteq L$, and $L$ and $K$ are conjugate.

(c) By parts (a) and (b), any closed inverse subsemigroup of $\S(\Gamma)$ that is conjugate to $L_{p,d}$ must be of cycle type.
Suppose that the closed inverse subsemigroups $L_{p,d}$ and $L_{q,k}$ are conjugate in $\S(\Gamma)$, and so
there exists $(s,t) \in \S(\Gamma)$ such that
\begin{align}
(t,s) L_{p,d} (s,t) \subseteq L_{q,k} \label{cong1}\\
(s,t) L_{q,k} (t,s) \subseteq L_{p,d}. \label{cong2}
\end{align}
Since $L_{q,k}$ is closed and $L_{p,d}$ is the smallest closed inverse subsemigroup of $\S(\Gamma)$ containing $(d,pd)$, then \eqref{cong1} is
equivalent to $(t,s)\,(d,pd)\,(s,t) \in L_{q,k}$.  Also, since $0 \not\in L_{q,k}$ we must have $s\,$ suffix-comparable with $u$ and $v$ whenever 
$(u,v)$ is an element of $L_{p,d}\,$.  Hence $(s,s) \in L_{p,d}\,$, and similarly $(t,t) \in L_{q,k}$.

First suppose that $s = up^ad\,$ and $t = vq^bk\,$ for some $a,b \geq 0$, where $u$ is a suffix of $p$ and $v$ is a suffix of $q$.  Write $p=hu$: then
\begin{align*}
(t,s)\,(d,pd)\,(s,t) &= (vq^bk,up^ad)\,(d,pd)\,(up^ad,vq^bk) \\
&= (vq^bk,up)\,(u,vq^bk) \\
&= (vq^bk,uhvq^bk) \in L_{q,k}.
\end{align*}
It follows that $uhvq^bk = vq^mk\,$ for some $m \geq 0$.  Comparing lengths of these directed paths, we see that $m>b$, and then  after cancellation
we obtain $uhv = vq^{m-b}$.  Hence $uh$ is conjugate to some power of $q$, and since $uh$ is a conjugate of $p$, we conclude that $p$ is conjugate to some power of $q$.

Now suppose that $s$ is a suffix of $d$ and write $d=cs$.  With $t$ as before, we now obtain
\[
(t,s)\,(d,pd)\,(s,t) = (vq^bk,s)\,(cs,pcs)\,(s,vq^bk) 
= (cvq^bk,pcvq^bk) \in L_{q,k}.
\]
It follows that $pcvq^bk = cvq^mk$ for some $m \geq 0$.  Again $m>b$ and after cancellation we obtain $pcv=cvq^{m-b}$.  Here we see directly that $p$ is conjugate to a power of $q$.

Now suppose that $s$ is a suffix of $d$ and write $d=cs$, and that $t$ is a suffix of $k$ and write $k=jt$.  We now obtain
\[
(t,s)\,(d,pd)\,(s,t) = (t,s)\,(cs,pcs)\,(s,t)
= (ct,pct) \in L_{q,k}.
\]
Since by assumption $p\,$ is not the empty path, we have $ct = wq^ak\,$ and $pct = wq^bk\,$ for some suffix $w$ of $q$ and some $a,b \geq 0$.  
Again comparing lengths, we see that $b>a$, and then $pct = pwq^ak = wq^bk$.  After cancellation we obtain $pw=wq^{b-a}$
and again $p$ is conjugate to a power of $q$.

Hence for each possibility of $s$, we deduce from \eqref{cong1} that $p\,$ is conjugate to some power of $q\,$.
Using equation \eqref{cong2} we deduce similarly that $q$ is conjugate to a power of $p\,$.  Again comparing lengths, we conclude that $p$ and $q$ are conjugate.

Conversely, if $p,q$ are conjugate, suppose that $p=uv$ and $q=vu$.  Then it is esay to check that setting $s=k$ and
$t=vd$ furnishes a pair $(s,t)$ satisfying \eqref{cong1} and \eqref{cong2}.
\end{proof}

\begin{remark}
\label{bi_poly}
For the polycyclic monoids $P_{n}$ $(n\geq 2)$, we obtain the classification of closed inverse submonoids up to conjugacy given in \cite[Theorem 4.4]{LwPoly} by applying Theorem \ref{cong_in_Gamma}  to the graph $\Gamma$ with one vertex and $n$ loops labelled $a_{1},\dots, a_{n}.$  For the case $n=1$, with a single loop labelled $a$, we obtain the graph inverse semigroup $\S(\Gamma)= B \cup \{0\}$, where $B$ is the bicyclic monoid.  A proper closed inverse subsemigroup $L$ of $\S(\Gamma)$ cannot contain $\{0\}$ and so is a proper closed inverse subsemigroup of $B$. If $L \subseteq E(B)$ then by Theorem \ref{clisms_of_grinv}, $L$ is either $E(B)$ itself or is of finite chain type,
and part (a) of Theorem \ref{cong_in_Gamma}, then shows that all  closed inverse subsemigroup of finite chain type in $B$ are then conjugate.  By Theorem \ref{clisms_of_grinv}, a closed inverse subsemigroup $L$ of $B$ of cycle type consists of elements of the form 
$(qp^r,qp^s)$ with $r,s\geq 0$, and where
where $p=a^m$ for some $m \geq 1$ and $q=a^k$ for some $k$ with $0 \leq k \leq m-1$:  that is, elements of the form 
$(a^{rm+k},a^{sm+k})$.  
The subsemigroup $L$ is therefore isomorphic to the fundamental  simple inverse $\omega$--semigroup $B_m$, discussed in
\cite[section 5.7]{HoBook}.
\end{remark}

\section{The index of closed inverse subsemigroups in graph inverse semigroups}
\label{clissindex}

We first discuss the index of closed inverse subsemigroups of finite and infinite chain type in $\S(\Gamma)$.  For a fixed path
$w$ in $\Gamma$ and a vertex $v$ of $w$, we define $N^{\Gamma}_{v,w}$ to be the number of distinct directed paths in
$\Gamma$ whose initial vertex
is $v$ but whose first edge is not in $w$.  The empty path $v$ is one such path.

\begin{theorem}
\label{idpt_type_grinv}
\leavevmode
\begin{enumerate}
\item
Let $L$ be a closed inverse subsemigroup of finite chain type in $\S(\Gamma)$, with minimal element $(w,w)$. Then $L$ has infinite index in $\S(\Gamma)$ if and only if there exists a non-empty directed circuit $c\,$ in $\Gamma$ and a (possibly empty) directed path $g$ from some vertex $v_0$ of $w$ to a vertex of $c\,$ and with $g$ having no edge in common with $c \cup w$.
\item If $L = \nobraupset{(w,w)}$ and $L$ has finite index in $\S(\Gamma)$ then
\[ [\S(\Gamma):L] = \sum_{v \in V(w)} N^{\Gamma}_{v,w} \,.\]
\item Let $L$ be a closed inverse subsemigroup of infinite chain type in $\S(\Gamma)$. Then $L$ has infinite index in $\S(\Gamma)$.
\end{enumerate}
\end{theorem}

\begin{proof}
(a) Let $L$ have finite chain type.  A coset representative of $L$ has the form $(q,u)$ where $q$ is some suffix of $w$, and $(q,u)$ have the same initial vertex.  If $L$ has infinite index, then there are infinitely many distinct choices for $(q,u)$ and since $\Gamma$ is finite, there must be a directed circuit in $\Gamma$ as described.

Conversely, suppose that $g,c\,$ exist.  Let $q\,$ be the suffix of $w\,$ that has initial vertex $v_0\,$.  Then $q \in L$ and so for any 
$k \geq 0\,$ the coset $C_k = L\upset{q,gc^k}$ exists. Now for $k>l,\,$ we have if $g$ is non-empty, that 
\[ (q,gc^k)(q,gc^l)^{-1} = (q,gc^k)(gc^l,q) = 0 \not\in L \]
and so by part (c) of Proposition \ref{cosetsofL},  the cosets $C_k\,$ and $C_l\,$ are distinct.  If $g$ is empty then we have
\[ (q,c^k)\,(q,c^l)^{-1} = (q,c^k)\,(c^l,q) = (q,c^{k-l}q) \not\in L \]
and again the cosets $C_k\,$ and $C_l\,$ are distinct.

(b) By part (a) there are no directed cycles accessible from any vertex of $w$, and so $N^{\Gamma}_{v,w}$ is finite for each vertex $v$
of $w$.  A coset representative of $L$ has the form $(s,t)$ where $s$ is a suffix of $w$.  Suppose that two such elements, 
$(s_1,t_1)$ and $(s_2,t_2)$, represent the same coset.  Then $(s_1,t_1)(t_2,s_2) \in L$: in particular the product is non-zero and so
$t_1, t_2$ are suffix comparable.  We may assume that $t_2=ht_1$: then $(s_1,t_1)(ht_1,s_2) = (hs_1,s_2)$ and this is in $L$
if and only if $s_2 = hs_1$.  Therefore $\nobraupset{L(s_1,t_1)} = \nobraupset{L(s_2,t_2)}$ if and only if $(s_2,t_2)=(hs_1,ht_1)$, and so the distinct coset representatives are the pairs $(s,t)$ where $s$ is a suffix of $w$, $s$ and $t$ have the same initial vertex, but do not 
share the same initial edge.  It follows that the number of distinct cosets is $\sum_{v \in V(w)} N^{\Gamma}_{v,w}$, and $L$ itself
is represented by $(\ran(w),\ran(w))$.

(c) If $L$ has infinite chain type then the elements of $L$ comprise  the idempotents determined by an infinite sequence of directed paths in $\Gamma,\,$ each of which is a suffix of the other.  Eventually then, we find a path $cq$ where $c$ is a directed circuit, and $(cq,cq) \in L$.  Then for each
$k \geq 0,\,$ the element $(q,c^kq)$ represents a coset $C_k = L\upset{q,c^kq}\,.$ \\ Now for $k>l,\,$
\[ (q,c^kq)\,(q,c^lq)^{-1} = (q,c^kq)\,(c^lq,q) = (q,c^{k-l}q) \not\in L \]
and the cosets $C_k\,$ and $C_l\,$ are distinct.
\end{proof}

\begin{example}
We illustrate the index computation in part (b) of Theorem \ref{idpt_type_grinv} with $\Gamma$ equal to the finite chain with
$n$ edges $e_1, \dotsc , e_n$ and $n+1$ vertices $v_0, v_1, \dotsc , v_n$:
\[ \xymatrixcolsep{3pc}
\xymatrix{
v_n  \ar[r]^{e_n}   & v_{n-1}  \ar[r]^{e_{n-1}} &  \dotsc \ar[r] & v_1 \ar[r]^{e_1} & v_0 }
\]
Here $\S(\Gamma)$ is finite, and every closed inverse subsemigroup is of finite cycle type and has finite index.
The number of paths in $\Gamma$ with initial vertex $v_j$ is $j+1$, and so
\[ |\S(\Gamma)| = \sum_{j=0}^n (j+1)^2 = \sum_{j=1}^{n+1} j^2 = \frac{1}{6}(n+1)(n+2)(2n+3) \,.\]
We let $w$ be the path $e_n \dotsm e_1$ and $L = \upset{w,w}$.  Since $w$ has $n+1$ suffixes, we have $|L|=n+1$.
An element $(s,t)$ lies in a coset of $L$ if and only if $s$ is a suffix of $w$ and $\dom(s)=\dom(t)$: hence the total number of
elements in all the cosets of $L$ is $\sum_{j=0}^n (j+1) = \sum_{j=1}^{n+1} j = \frac{1}{2}(n+1)(n+2)$.  

Now $N^{\Gamma}_{v_i,w} = 1$ since only the length zero path at $v_i$ is counted, and so $[\S(\Gamma):L] = n+1$.

Let $q_i$ be the path $e_i \dotsc e_1$, so that $q_n = w$, and set $q_0 = v_0$.
The $n+1$ cosets are then represented by the elements $(q_i,v_i)$, $0 \leq i \leq n$, and 
\begin{align*} \nobraupset{L(q_i,v_i)} &= \nobraupset{\{ (q_k,q_k)(q_i,v_i) : 0 \leq k \leq n \}} \\
&= \upset{\{ (q_k,e_k \dotsm e_{i+1} : i < k \leq n \} \cup \{ (q_i,v_i) \}} \\
&= \{ (q_n, e_n \dotsc e_{i+1}), \dotsc , (q_{i+1},e_{i+1}),(q_i,v_i) \}
\end{align*}
and so $|\nobraupset{L(q_i,v_i)}|=n-i+1$.  Counting the total number of elements in all the cosets of $L$ we obtain
\[ \sum_{i=0}^n (n-i+1) = \sum_{j=1}^{n+1} \,  \frac{1}{2}(n+1)(n+2) \]
as before.
\end{example}

We now discuss the closed inverse subsemigroups of cycle type.

\begin{theorem}
\label{cycle_type_grinv}
A closed inverse subsemigroup\ $L_{p,d}$ of cycle type in $\S(\Gamma)$, such that $p$ is a circuit with at least two distinct edges, has infinite index in $\S(\Gamma)$.
\end{theorem}
   
\begin{proof}
Write $p=uv$ where each of $u,v$ is non-empty and one contains an edge not in the other. Let $c$ be the conjugate circuit $vu$.
Then for $k \geq 1,\,$ the element
$(vd,c^k) \in \S(\Gamma)$ and determines a coset $C_k = L_{p,d}\upset{vd,c^k}$.  Then for $k>l,\,$
\[ (vd,c^k)\,(vd,c^l)^{-1} = (vd,c^k)\,(c^l,vd) = (vd,c^{k-l}vd) = (vd,up^{k-l}d) \not\in L_{p,d} \]
and the cosets $C_k$ and $C_l$ are distinct. 
\end{proof}

We now consider a graph $\Gamma$ containing an edge $a$ that is a directed circuit of length one, and a closed inverse subsemigroup $L_{a^m,\,d}$ of cycle type.

\begin{theorem}
\label{loops}
\leavevmode
\begin{enumerate}
\item
A closed inverse subsemigroup\ $L = L_{a^m,\,d}$ of $\S(\Gamma)$ is of infinite index if there exists a directed cycle $c$ in $\Gamma$
that contains an edge $e$ with $a \ne e$, and a (possibly empty) directed path $g$ from some vertex $v_0$ of $d$ to a vertex of $c$ and with $g$ having no edge in common with $c \cup d$. 
\item  Let $L=L_{a^m,\,d}$ where $a$ is a directed circuit in $\Gamma$ of length one, and there are no other directed circuits in
$\Gamma$ attached to a vertex of $d$.  Then $L$ has finite index in $\S(\Gamma)$, given by
\[ [\S(\Gamma):L_{a^m,d}] = (m-1)N^{\Gamma}_{\dom(a),a} + \sum_{v \in V(d)} N^{\Gamma \setminus \{ a \}}_{v,d}\,. \]
\end{enumerate} 
\end{theorem}

\begin{proof}
(a) Suppose that $c,g$ exist and let $q$ be the suffix of $d$ with initial vertex $v_0$. 

Let 
$C_k = L\upset{q,gc^k}\,$.  Then if $g$ is non-empty, for $k>l$,
\[ (q,gc^k)\,(q,gc^l)^{-1} = (q,gc^k)\,(gc^l,q) = 0 \not\in L \]
and the cosets $C_k$ and $C_l$ are distinct.  If $g$ is empty, then
\[ (q,c^k)\,(q,c^l)^{-1} = (q,c^k)\,(c^l,q) = (q,c^{k-l}q)  \not\in L \]
and the cosets $C_k\,$ and $C_l\,$ are again distinct.

(b) We are now reduced to the case that the only directed circuits in $\Gamma$ that can be attached to a vertex of $d$ are powers of the loop $a$.
A coset representative of $L=L_{a^m,d}$ has the form $(a^rd,w)$ with $r \geq 0$, or $(q,w)$ where $q$ is a proper
suffix of $d$.  Hence $w$ has the same initial vertex $v$ as $d$ or of some proper suffix of $d$.  We can only construct finitely many 
representatives of the form $(q,w)$.  We do not need to consider paths of the form $(d,a^kw)$ for any $k \geq 0$ since
$(d,a^kw)(w,d) = (d,a^kd) \in L$.  The analysis in the proof of part (b) of Theorem  \ref{idpt_type_grinv} can then be repeated to
show that the number of cosets obtained this way is $\sum_{v \in V(d)} N^{\Gamma \setminus \{ a \}}_{v,d}$.  

We now consider representatives of the form $(a^rd,w)$ with $r \geq 1$.  Here $w$ must have the form
$w=a^st$ for some $s \geq 0$ and some (possibly empty) directed path $t$ not containing the edge $a$. If
$r \equiv s \pmod{m}$ then 
$(a^rd,a^st)(t,d) = (a^rd,a^sd) \in L$
and so $\nobraupset{L(a^rd,a^st)} = \nobraupset{L(d,t)}= \nobraupset{L(s,t_1)}$ for some suffix $s$ of $d$ and path $t_1$ with the 
same initial vertex as $s$ but not sharing the same first edge.  Hence  $\nobraupset{L(a^rd,a^st)}$ will be counted within the 
sum $\sum_{v \in V(d)} N^{\Gamma \setminus \{ a \}}_{v,d}$.  Now fix $t$ and consider  the cosets $\nobraupset{L(a^rd,a^st)}$
with $r \not\equiv s \pmod{m}$.  Now given
$L\upset{a^{r_1}d,a^{s_1}t}$ and $L\upset{a^{r_2}d,a^{s_2}t}$ with $s_1 \geq s_2$, we have
\[ (a^{r_1}d,a^{s_1}t)(a^{r_2}d,a^{s_2}t)^{-1} = (a^{r_1}d,a^{s_1}t)(a^{s_2}t,a^{r_2}d) = (a^{r_1}d,a^{s_1-s_2+r_2}d) \]
and $(a^{r_1}d,a^{s_1-s_2+r_2}d) \in L$ if and only if  $r_2-s_2 \equiv r_1-s_1 \pmod{m}$. Hence for a fixed $t$ we can
produce exactly $m-1$ distinct cosets of the form $L\upset{a^rd,a^st}$.

But for distinct paths $t_1$ and $t_2$, $a^{s_1}t_1$ cannot be suffix comparable with $a^{s_2}t_2$ and so
\[ (a^{r_1}d,a^{s_1}t_1)(a^{r_2}d,a^{s_2}t_2)^{-1} = 0 \ne L \]
and the cosets determined by distinct paths $t_1$ and $t_2$ are distinct.  Hence each of the $N^{\Gamma}_{\dom(a),a}$
paths $t$ starting at $\dom(a)$, but not having $a$ as its initial edge, contributes $m-1$ cosets.
\end{proof}

\begin{example}
\label{bicyclic_with_zero}
As in Remark \ref{bi_poly}, we suppose that $\Gamma$ consists only of the vertex $x$ and a loop $a$ at $x$ so that the graph inverse semigroup $\S(\Gamma)$ is the 
bicyclic monoid $B$ with a zero adjoined.
From Theorem
\ref{idpt_type_grinv}, the closed inverse submonoids of $B$ contained in $E(B)$ have infinite index. Part (b) of Theorem \ref{loops}
tells us that that the closed inverse submonoid $B_m=L_{a^m,x}$ of $B$ has index $m$.
\end{example}

\begin{example}
\label{cycle_type_eg}
Let $\Gamma$ be the following graph:
\[ \xymatrixcolsep{3pc}
\xymatrix{
x' \ar[r]^h & y' \\
x \ar@(ul,dl)[]_a \ar[r]^{e} \ar[u]_g  & y  \ar[r]^{f} \ar[u]_k &  z }
\]
and let $L = L_{a^2,ef}$.  Then we have
\[ N^{\Gamma \setminus a}_{z,ef} = 1 \,,  N^{\Gamma \setminus a}_{y,ef} = 2 \,, N^{\Gamma \setminus a}_{x,ef} = 3 \,,\]
counting the paths in the sets $\{z \}$ , $\{ y,k\}$ and $\{ x,g,gh \}$ respectively, and $N^{\Gamma}_{x,a} = 6$, counting the 
paths in the set $\{x,e,g,ef,ek,gh\}$.  From part (b) of Theorem \ref{loops} we find that $[\S(\Gamma),L]=12$ and a complete set of
coset representatives is 
\[ \left\{(z,z),(f,y),(f,k),(ef,x),(ef,g),(ef,gh), \right.\]
\[ \left. (ef,a),(ef,ag),(ef,agh),(ef,ae),(ef,aek),(ef,aef) \right\} \,.\]
\end{example}

\end{document}